\documentclass[12pt]{article}

\usepackage[margin=1in]{geometry}

\usepackage[utf8]{inputenc}
\usepackage{amssymb}
\usepackage{amsmath}
\usepackage{amsthm}
\usepackage{url}
\usepackage{MnSymbol}

\usepackage{hyperref}

\newtheorem{theorem}{Theorem}[section]
\newtheorem{lemma}[theorem]{Lemma}

\newtheorem{corollary}[theorem]{Corollary}
\newtheorem{conjecture}[theorem]{Conjecture}

\newtheorem*{claimstar}{Claim}

\theoremstyle{definition}
\newtheorem{definition}[theorem]{Definition}
\newtheorem{example}[theorem]{Example}

\newtheorem{question}[theorem]{Question}

\theoremstyle{remark}

\newcommand{\mc}[1]{\mathcal{#1}}

\newcommand{\fneg}{\mathord{\sim}}

\DeclareMathOperator{\PC}{PC}

\newcommand{\ba}{\bold{a}}

\begin{document}

\title{Which classes of structures are both pseudo-elementary and definable by an infinitary sentence?}
\author{Will Boney, Barbara F. Csima\thanks{Partially supported by Canadian NSERC Discovery Grant 312501.}, Nancy A. Day, Matthew Harrison-Trainor\thanks{Supported by an NSERC Banting Fellowship.}}

\makeatletter
\def\blfootnote{\xdef\@thefnmark{}\@footnotetext}
\makeatother

\blfootnote{This work grew out of initial discussions with Vakili about the generality of expressing properties not definable in first-order logic in a pseudo-elementary way, and whether such phenomena might be of use for model checking (as the pseudo-elementary definability of graph reachability was used for model checking by Vakili in his thesis \cite{Vakili16} and with the third author in \cite{VakiliDay14}). We thank one of the referees for pointing us towards some very helpful references.}

\maketitle

\begin{abstract}
When classes of structures are not first-order definable, we might still try to find a nice description. There are two common ways for doing this. One is to expand the language, leading to notions of pseudo-elementary classes, and the other is to allow infinite conjuncts and disjuncts. In this paper we examine the intersection. Namely, we address the question: Which classes of structures are both pseudo-elementary and $\mc{L}_{\omega_1, \omega}$-elementary? We find that these are exactly the classes that can be defined by an infinitary formula that has no infinitary disjunctions.
\end{abstract}

\section{Introduction}

It is well-known that many properties are not definable in elementary first-order logic, even by a theory rather than a single sentence. Common examples are the property (of graphs) of being connected, the property (of abelian groups) of being torsion, and the property (of linear orders) of being well-founded. To capture such properties, one can pass to extensions of elementary first-order logic. This paper is about a characterization of the common expressive power of two such extensions.

The first extension of elementary first-order logic that we consider is to allow countably infinite conjunctions and disjunctions; this is, morally, similar to allowing quantifiers over the (standard) natural numbers. One can then define properties such as being torsion by saying ``for each group element $x$, there is $n$ such that $nx = 0$'', or formally,
\[ (\forall x) \bigdoublevee_{n \in \mathbb{N}} nx = 0.\]
This infinitary logic is known as $\mc{L}_{\omega_1, \omega}$. One loses compactness, but gains other powerful tools. For example, every countable structure is described, up to isomorphism among countable structures, by a sentence of $\mc{L}_{\omega_1, \omega}$ \cite{Scott65}.

The second extension of elementary first-order logic is to allow existential second-order quantifiers. For example, the property of a linear order being non-well-founded can be defined by the sentence ``there is a set with no least element''. We say that such a property is pseudo-elementary. More formally, a property $P$ of $\tau$-structures is pseudo-elementary if there is an expanded language $\tau^* \supseteq \tau$ and an $\tau^*$-sentence $\varphi$ (or $\tau^*$-theory $T$) such that the $\tau$-structures admitting an $\tau^*$-expansion to a model of $\varphi$ (respectively $T$) are exactly the structures satisfying $P$. We will describe both of these extensions of first-order logic in more detail later.

These two extensions of elementary first-order logic have different descriptive powers. For example, the property of being non-well-founded is pseudo-elementary but not $\mc{L}_{\omega_1, \omega}$-definable. Also, the negation of a pseudo-elementary property is not necessarily pseudo-elementary, but the negation of an $\mc{L}_{\omega_1, \omega}$-definable property is again $\mc{L}_{\omega_1, \omega}$-definable. Nevertheless, there are properties that are not elementary first-order definable, but that are both pseudo-elementary and $\mc{L}_{\omega_1, \omega}$-definable. The property of a graph being disconnected is such an example; we provide a more detailed discussion of various examples in Section \ref{sec:ex}. The main result of this paper is a complete classification of such properties.

\begin{theorem}\label{thm:main}
Let $\mathbb{K}$ be a class of structures closed under isomorphism. The following are equivalent:
\begin{enumerate}
	\item $\mathbb{K}$ is both a pseudo-elementary \textup{(}$\PC_{\Delta}$\textup{)} class and defined by an $\mc{L}_{\omega_1, \omega}$-sentence.
	\item $\mathbb{K}$ is defined by a $\bigdoublewedge$-sentence,
\end{enumerate}
\end{theorem}

\noindent Theorem \ref{thm:main} follows immediately from Corollary \ref{cor:both-to-pc} (for (1)$\Rightarrow$(2)) and Theorem \ref{thm:2-1} (for (2)$\Rightarrow$(1)).

There is some notation in this theorem that we must explain. First, there are some subtleties in the definition of what it means for a property to be pseudo-elementary, and in fact there are four different natural definitions (giving rise to three distinct notions; these are discussed in Section \ref{pec-ssec}). Two of them are as follows.

\begin{definition}
We say that a class $\mathbb{K}$ of $\tau$-structures is a $\PC$-class if there is a language $\tau^* \supseteq \tau$ and an elementary first-order $\tau^*$-sentence $\phi$ such that
\[ \mathbb{K} = \{ \mc{M} \mid \text{there is an $\tau^*$-structure $\mc{M}^*$ expanding $\mc{M}$ with $\mc{M}^* \models \phi$}\}.\]
We say that $\mathbb{K}$ is a $\PC_\Delta$-class if $\phi$ is replaced by an elementary first-order theory.
\end{definition}

\noindent So the theorem above is concerned with pseudo-elementary classes where one is allowed to use a theory in the definition.

The $\bigdoublewedge$-sentences in the theorem are the $\mc{L}_{\omega_1, \omega}$-sentences which (in normal form) involve infinitary conjunctions, but no infinitary disjunctions. For example, the property of being infinite is definable by the $\bigdoublewedge$-sentence
\[ \bigdoublewedge_{n \in \mathbb{N}} \exists x_1,\ldots,x_n (\bigwedge_{i \neq j} x_i \neq x_j). \]
The negation, the property of being finite, is $\mc{L}_{\omega_1, \omega}$-definable by the sentence
\[ \bigdoublevee_{n \in \mathbb{N}} \forall x_1,\ldots,x_n (\bigvee_{i \neq j} x_i = x_j) \]
but this sentence is not a $\bigdoublewedge$-sentence because it involves an infinitary disjunct.

\begin{definition}
The $\bigdoublewedge$-formulas are defined inductively as follows:
\begin{itemize}
	\item every finitary quantifier-free formula is a $\bigdoublewedge$-formula
	\item if $\varphi$ is a $\bigdoublewedge$-formula, then so are $(\exists x) \varphi$ and $(\forall x) \varphi$
	\item if $(\varphi_i)_{i \in \omega}$ are $\bigdoublewedge$-formulas with finitely many free variables, then so is $\bigdoublewedge_{i \in \omega} \varphi_i$.
\end{itemize}
\end{definition}

The proof of (1)$\Rightarrow$(2) uses an argument inspired by the proof of Craig Interpolation for $\mc{L}_{\omega_1, \omega}$. This was originally proved by Lopez-Escobar \cite{LopezEscobar65} who also gave the following corollary: a property which is both pseudo-elementary and co-pseudo-elementary with respect to $\mc{L}_{\omega_1, \omega}$ (i.e., both $\mathbf{\Sigma}^1_1$ and $\mathbf{\Pi}^1_1$) is actually $\mc{L}_{\omega_1, \omega}$-definable.

In the direction (2)$\Rightarrow$(1), there are several possible proofs. We give the simplest and shortest argument in Section \ref{sec:simple}. A second proof is to note that any $\bigdoublewedge$-sentence is equivalent to a closed game formula, and these are known to $\PC_\Delta$ \cite{Kolaitis}. We describe this in Section \ref{sec:game}.
\noindent A third proof, for which we do not give the details, proceeds by coding computable formulas in models of weak arithmetic. This is an approach that was taken by Craig and Vaught \cite{CraigVaught} to prove:
\begin{theorem}[Craig and Vaught \cite{CraigVaught}]\label{cor:comp ax gives pc'}
	Every computably axiomatizable class in a finite language is pseudo-elementarily defined using a single sentence \textup{(}$PC'$\textup{)}.
\end{theorem}
A result on closed game formulas in \cite{Barwise} gives:
\begin{theorem}\label{thm:comp}
	Let $\mathbb{K}$ be a class of structures in a finite language that is axiomatized by a computable $\bigdoublewedge$-sentence. Then $\mathbb{K}$ is pseudo-elementarily defined using a single sentence \textup{(}$PC'$\textup{)}.
\end{theorem}

\noindent Indeed, Theorem \ref{cor:comp ax gives pc'} is a corollary of Theorem \ref{thm:comp}. Unfortunately, we do not know how to reverse Theorem \ref{thm:comp}. We conjecture:
\begin{conjecture}
A $\PC$ or $\PC'$ class which is also $\mc{L}_{\omega_1, \omega}$-axiomatizable is axiomatizable by a computable $\bigdoublewedge$-sentence.
\end{conjecture}

The argument in Section \ref{sec:simple} for (2)$\Rightarrow$(1) of Theorem \ref{thm:main} goes through for $\bigdoublewedge$-sentences of $\mc{L}_{\kappa,\omega}$ for any $\kappa$. However, we do not know if (1)$\Rightarrow$(2) holds for $\mc{L}_{\kappa,\omega}$ for $\kappa> \omega_1$.
\begin{question} \label{beyond-quest}
	For $\kappa > \omega_1$, is every $\PC_\Delta$ class defined by an $\mc{L}_{\kappa,\omega}$ sentence actually defined by a $\bigdoublewedge$-sentence?
\end{question}
\noindent We note that interpolation fails in $\mc{L}_{\omega_2,\omega}$ \cite[Theorem 4.2]{m-infanal}.  Intriguingly, Malitz goes on to give a proof system for $\mc{L}_{\kappa, \omega}$ that goes through $\mc{L}_{(2^{<\kappa})^+, \kappa}$ that gives rise to an interpolation theorem \cite[Section 5]{m-infanal}.  Shelah \cite{sh797} uses this to define a logic $\mc{L}^1_\kappa$ that is intermediate between $\mc{L}_{\kappa, \omega}$ and $\mc{L}_{\kappa, \kappa}$ that has interpolation and other nice properties (when $\kappa = \beth_\kappa$).  This suggests the right answer to Question \ref{beyond-quest} goes through $\mc{L}^1_\kappa$ instead of $\mc{L}_{\kappa, \omega}$.  However, this logic lacks any syntax in the normal sense (formulas are defined by the existence of winning strategies in a delayed Ehrenfeucht-Fraisse game), which causes additional problems, e.g., it is not clear what a $\bigdoublewedge$-sentence should mean, or what Skolem functions should look like.

\section{Notation and Definitions}\label{sec:nota-def}

\subsection{Infinitary Logic}

For the most part, we follow Marker's book \cite{Marker16}. We want to be precise with our definitions here, because we will need to encode infinitary formulas in first-order sentences. We first define $\mc{L}_{\omega_1, \omega}$-formulas. Throughout the paper, let $\tau$ be a countable language.

\begin{definition}
The $\mc{L}_{\omega_1, \omega}(\tau)$-formulas are defined inductively as follows:
\begin{itemize}
	\item every atomic $\tau$-formula is an $\mc{L}_{\omega_1, \omega}(\tau)$-formula,
	\item if $\varphi$ is an $\mc{L}_{\omega_1, \omega}(\tau)$-formula, then so are $\neg \varphi$, $(\exists x) \varphi$ and $(\forall x) \varphi$,
	\item if $(\varphi_i)_{i \in \omega}$ are $\mc{L}_{\omega_1, \omega}(\tau)$-formulas with finitely many free variables, then so are $\bigdoublewedge_{i \in \omega} \varphi_i$ and $\bigdoublevee_{i \in \omega} \varphi_i$.
\end{itemize}
\end{definition}

\noindent In general, we will drop the reference to $\tau$ when it is clear what we mean.

\begin{definition}
An $\mc{L}_{\omega_1, \omega}$-formula is in \emph{$\mc{L}_{\omega_1, \omega}$ normal form} if the $\neg$ only occurs applied to atomic formulas.
\end{definition}

\noindent Every $\mc{L}_{\omega_1, \omega}$-fromula can be placed into a normal form. The negation $\neg \varphi$ of a sentence $\varphi$ in normal form is not immediately in normal form itself. This gives rise to the formal negation $\fneg \varphi$, which is logically equivalent to $\neg \varphi$ but is in normal form.

\begin{definition}
For any $\mc{L}_{\omega_1, \omega}$-formula $\varphi$, the formula $\fneg \varphi$ is defined inductively as follows:
\begin{itemize}
	\item if $\varphi$ is atomic, $\fneg \varphi$ is $\neg \varphi$,
	\item $\fneg \neg \varphi$ is $\varphi$, $\fneg (\exists x) \varphi$ is $(\forall x) \fneg \varphi$ and $\fneg (\forall x) \varphi$ is $(\exists x) \fneg \varphi$,
	\item $\fneg \bigdoublewedge_{i \in \omega} \varphi_i$ is $\bigdoublevee_{i \in \omega} \fneg \varphi_i$ and $\fneg \bigdoublevee_{i \in \omega} \varphi_i$ is $\bigdoublewedge_{i \in \omega} \fneg \varphi_i$.
\end{itemize}
\end{definition}

\noindent We repeat again the definition of a $\bigdoublewedge$-formula.

\begin{definition}
An $\mc{L}_{\omega_1, \omega}$-formula $\varphi$ is a $\bigdoublewedge$-formula if it can be written in normal form without any infinite disjunctions. More concretely, the $\bigdoublewedge$-formulas are defined inductively as follows:
\begin{itemize}
	\item every finitary quantifier-free formula is a $\bigdoublewedge$-formula,
	\item if $\varphi$ is a $\bigdoublewedge$-formula, then so are $(\exists x) \varphi$ and $(\forall x) \varphi$,
	\item if $(\varphi_i)_{i \in \omega}$ are $\bigdoublewedge$-formulas with finitely many free variables, then so is $\bigdoublewedge_{i \in \omega} \varphi_i$.
\end{itemize}
\end{definition}

An $\mc{L}_{\omega_1, \omega}$ (or $\bigdoublewedge$-) formula is computable if, essentially, there is a computable syntactic representation of the formula (see \cite{AshKnight00}).

\subsection{Pseudo-elementary Classes} \label{pec-ssec}

In this section, we follow the book by Hodges \cite{Hodges93}. There are four different types of pseudo-elementary classes: $\PC$, $\PC'$, $\PC_\Delta$, and $\PC_\Delta'$. The $\Delta$ means that we are allowed a full theory rather than a single sentence, and the ${}^\prime$ means that we are allowed to add additional sorts (elements) to the structure. The classes $\PC$ and $\PC_{\Delta}$ have already been defined, but we repeat the definition.

\begin{definition}
We say that a class $\mathbb{K}$ of $\tau$-structures is a $\PC$-class if there is a language $\tau^* \supseteq \tau$ and an elementary first-order $\tau^*$ sentence $\phi$ such that
\[ \mathbb{K} = \{ \mc{M} \mid \text{there is a $\tau^*$-structure $\mc{M}^*$ expanding $\mc{M}$ with $\mc{M}^* \models \phi$}\}.\]
We say that $\mathbb{K}$ is a $\PC_\Delta$-class if $\phi$ is replaced by an elementary first-order theory.
\end{definition}

The classes $\PC'$ and $\PC'_\Delta$ are a little more complicated to define. We need the following definition, which one should think of as throwing away a sort from a structure.

\begin{definition}
Let $\tau \subseteq \tau^*$ be a pair of languages, with a unary predicate $P \in \tau^* \setminus \tau$. Given a $\tau^*$-structure $\mc{A}$, we denote by $\mc{A}_P$ the substructure of $\mc{A} \mid \tau$ whose domain is $P^{\mc{A}}$ (if this is a $\tau$-structure; otherwise $\mc{A}_P$ is not defined).
\end{definition}

\noindent The classes $\PC'$ and $\PC'_\Delta$ differ from $\PC$ and $\PC_{\Delta}$ respectively in that in addition to expanding the language, one is allowed to add additional elements.

\begin{definition}
We say that a class $\mathbb{K}$ of $\tau$-structures is a $\PC'$-class if there is a language $\tau^* \supseteq \tau$, with a unary relation $P \in \tau^* \setminus \tau$, and a $\tau^*$-formula $\phi$, such that
\[ \mathbb{K} = \{ \mc{A}_P \mid \text{$\mc{A} \models \phi$ and $\mc{A}_P$ is defined}\}.\]
We say that $\mathbb{K}$ is a $\PC'_\Delta$-class if $\phi$ is a first-order theory.
\end{definition}

\noindent Note that, if the language is finite (or we are dealing with a $\PC'_\Delta$-class) it suffices to ask that
\[ \mathbb{K} = \{ \mc{A}_P \mid \text{$\mc{A} \models \phi$}\} \]
as $\phi$ can say that $\mc{A}_{P}$ is defined.

Though we have four different definitions, they give rise to only three different notions (and only two if we consider classes which consist only of infinite structures).

\begin{theorem}[Theorem 5.2.1 of \cite{Hodges93}]\label{thm:hodges}
Let $\mathbb{K}$ be a class of structures.
\begin{itemize}
	\item $\mathbb{K}$ is a $\PC_{\Delta}$-class if and only if it is a $\PC'_{\Delta}$-class.
	\item If all the structures in $\mathbb{K}$ are infinite, then $\mathbb{K}$ is a $\PC$-class if and only if it is a $\PC'$-class.
\end{itemize}
\end{theorem}

\noindent In Example \ref{ex:pc'} we give a class which is $\PC'$ but not $\PC$.

 The proof of the first point in \cite{Hodges93} is not obvious and quite interesting. For the second, essentially the only reason that $\PC$ and $\PC'$ are different is that the model might be finite; if a model is infinite, one could just have the elements of the model ``wear two hats'', on the one hand being the domain of the expansion of the original model, and on the other hand playing the role of the elements of the new sort $P$.

\section{Examples}\label{sec:ex}

We give here a few examples of properties that are definable in various combinations of expansions of elementary first-order logic, including some applications of the theorems.

\begin{example}\label{ex:graphs}
Let $\tau = \{ R \}$ the language of graphs.
The class $\mathbb{K}$ of non-connected graphs is a $\PC$-class.
Indeed, an undirected graph $\mc{G} = (G,R)$ is disconnected if and only if there is a binary relation $C$ of connectedness such that
\begin{itemize}
	\item $(\forall x)(\forall y) [R(x,y) \rightarrow C(x,y)]$,
	\item $(\forall x)(\forall y)(\forall z)[ C(x,y) \wedge C(y,z) \rightarrow C(x,z)]$, and
	\item $\neg (\forall x)(\forall y) C(x,y)$.
\end{itemize}
An undirected graph $\mc{G}$ is also disconnected if and only if
\[ (\exists x \neq y) \bigdoublewedge_{n \in \omega} (\forall u_0,\ldots,u_n)[x \neq u_0 \vee \neg R(u_0,u_1) \vee \neg R(u_1,u_2) \vee \cdots \vee \neg R(u_{n-1},u_n) \vee u_n \neq y ].\]
So $\mathbb{K}$ is also defined by a $\bigdoublewedge$-sentence.
\end{example}

\begin{example}
Let $\tau = \{ < \}$ the language of linear orders.
The class $\mathbb{K}$ of non-well-founded linear orders is a $\PC$-class as
a linear order $(S,<)$ is non-well-founded if and only if there is a unary relation $U$ such that
\[ (\forall x)\big[ x \in U \rightarrow (\exists y)[y \in U \; \wedge \; y < x]\big].\]
$\mathbb{K}$ is not definable by any $\mc{L}_{\omega_1, \omega}$ formula.
\end{example}

\begin{example}
Let $\tau$ be any language and $\phi$ a $\tau$-sentence.
The class $\mathbb{K}$ of infinite models of $\phi$ is a $\PC$-class as $\mc{A} \models \phi$ is infinite if and only if there is a linear order $<$ on $\mc{A}$ such that $(\forall x)(\exists y) [ x < y]$.
$\mathbb{K}$ is also defined by the $\bigdoublewedge$-sentence
\[ \bigdoublewedge_{n \in \omega} (\exists x_0,\ldots,x_n) \left[ \bigwedge_{i \neq j} x_i \neq x_j \right].\]
\end{example}

\begin{example}\label{og-ex}
Orderable groups are a PC-class. They are also universally axiomatizable (in first-order logic) by saying that every finite subset can be ordered in a way that is compatible with the group operation.
\end{example}

Example \ref{og-ex} is a particular instance of a more general phenomena: if we take a $\PC$-class that such that (a) the expanded vocabulary only adds relations and (b) the added relations are only universally quantified over, then the resulting class is actually elementary (though it may require infinitely many axioms).

\begin{example}\label{ex:pc'}
There is a c.e.\ universal theory $T$ whose models do not form a $\PC$-class; by Theorem \ref{cor:comp ax gives pc'} they are, however, a $\PC'$-class. The language of $T$ will be the language of graphs. Fix an enumeration of the sentences $\phi_n$ in finite languages $\mc{L}_n$ expanding the language of graphs. Note that for every finite graph $G$, we can decide effectively whether there is an expansion of $G$ to a model of $\phi_n$. For each $n$, let $C_n$ be cycle of length $n$. Then, let $T$ be the theory that says that there is no cycle of length $n$ for exactly those $n$ where $C_n$ does not have an expansion to a model of $\phi_n$. Note that $T$ is c.e.\ and universal, and that it is different from each $\PC$-class.
\end{example}

%

\section{An Application of Craig Interpolation}\label{sec:craig}

To prove the direction (1) implies (2) of Theorem \ref{thm:main}, we will adapt a proof of the Craig Interpolation Theorem for $\mc{L}_{\omega_1, \omega}$. The proof we adapt is not the original proof by Lopez-Escobar, but one that appears in the book by Marker \cite{Marker16}. We begin with a few preliminaries.

\begin{lemma}\label{lem:disj}
If $\varphi_1$ and $\varphi_2$ are $\bigdoublewedge$-formulas, $\varphi_1 \vee \varphi_2$ is equivalent to a $\bigdoublewedge$-formula.
\end{lemma}
\begin{proof}
We argue by induction on the complexity of $\varphi_1$ and $\varphi_2$ together. If $\varphi_1$ and $\varphi_2$ are both finitary quantifier-free, then so is $\varphi_1 \vee \varphi_2$. For the inductive steps, we will give the argument which reduces the complexity of $\varphi_1$; the arguments for $\varphi_2$ are similar.

If $\varphi_1$ is of the form $(Q x) \varphi_1'(x)$, where $Q$ is either $\exists$ or $\forall$, then $\varphi_1 \vee \varphi_2$ is equivalent to $(Q v)[\varphi_1'(v) \vee \varphi_2]$ where $v$ is not free in $\varphi_2$; by the inductive hypothesis, $\varphi_1'(v) \vee \varphi_2$ is equivalent to a $\bigdoublewedge$-formula and so $(Q v)[\varphi_1'(v) \vee \varphi_2]$ is as well.

Finally, if $\varphi_1$ is of the form $\bigdoublewedge_{\phi \in X} \phi$, where each $\phi$ is a $\bigdoublewedge$-formula, then $\varphi_1 \vee \varphi_2$ is equivalent to $\bigdoublewedge_{\phi \in X} [\phi \vee \varphi_2]$, and by the induction hypothesis, each $\phi \vee \varphi_2$ is a $\bigdoublewedge$-formula.
\end{proof}

The proof of Craig Interpolation makes use of consistency properties. Consistency properties are the infinitary equivalent of Henkin-style constructions in finitary logic. The following definition, due to Makkai, is what we need to do to perform such a construction.

\begin{definition}[Definition 4.1 of \cite{Marker16}]
Let $C$ be a countable collection of new constants. A consistency property $\Sigma$ is a collection of countable sets $\sigma$ of $\mc{L}_{\omega_1, \omega}$-sentences with the following properties. For $\sigma \in \Sigma$:
\begin{enumerate}
\item if $\mu \subseteq \sigma$, then $\mu \in \Sigma$;
\item if $\phi \in \sigma$, then $\neg \phi \notin \sigma$;
\item if $\neg \phi \in \sigma$, then $\sigma \cup \{ \sim \phi\} \in \Sigma$;
\item if $\bigdoublewedge_{\phi \in X} \phi \in \sigma$, then for all $\phi \in X$, $\sigma \cup \{\phi\} \in \Sigma$;
\item if $\bigdoublevee_{\phi \in X} \phi \in \sigma$, then there is $\phi \in X$ such that $\sigma \cup \{\phi\}  \in \Sigma$;
\item if $(\forall v) \phi(v) \in \sigma$, then for all $c \in C$, $\sigma \cup \{\phi(c)\}  \in \Sigma$;
\item if $(\exists v) \phi(v) \in \sigma$, then there is $c \in C$ such that $\sigma \cup \{\phi(c)\} \in \Sigma$;
\item let $t$ be a term with no variables and let $c,d \in C$,
\begin{enumerate}
\item if $c = d \in \sigma$, then $\sigma \cup \{d = c\}  \in \Sigma$;
\item if $c = t \in \sigma$ and $\phi(t) \in \sigma$, then $\sigma \cup \{\phi(c)\} \in \Sigma$;
\item there is $e \in C$ such that $\sigma \cup \{e = t\} \in \Sigma$.
\end{enumerate}
\end{enumerate}
\end{definition}

\noindent A consistency property is in some sense a recipe for building a model.

\begin{theorem}[Model Existence Theorem; see Theorem 4.1.6 of \cite{Marker16}]
If $\Sigma$ is a consistency property and $\sigma \in \Sigma$, there is $\mc{M} \models \sigma$.
\end{theorem}

We are now ready to prove our variant of the Craig Interpolation Theorem. We strengthen the hypotheses to assume that one of the sentences is a $\bigdoublewedge$-sentence, and in return, we get that the interpolant is also a $\bigdoublewedge$-sentence. The proof follows the same structure as that of the Craig Interpolation Theorem in \cite{Marker16} (Theorem 4.3.1).

\begin{theorem}
Suppose $\phi_1$ is a $\bigdoublewedge$-sentence and $\phi_2$ is an $\mc{L}_{\omega_1, \omega}$-sentence with $\phi_1 \models \phi_2$. There is a $\bigdoublewedge$-sentence $\theta$ such that $\phi_1 \models \theta$, $\theta \models \phi_2$, and every relation, function and constant symbol occurring in $\theta$ occurs in both $\phi_1$ and $\phi_2$.
\end{theorem}
\begin{proof}
Let $C$ be a countable collection of new constants. Let $\tau_i$ be the smallest language containing $\phi_i$ and $C$, and let $\tau = \tau_1 \cap \tau_2$.

Let $\Sigma$ be the collection of finite $\sigma$ containing only finitely many new constants that can be written as $\sigma = \sigma_1 \cup \sigma_2$, where $\sigma_1$ is a finite set of $\bigdoublewedge$-$\tau_1$-sentences and $\sigma_2$ is a finite set of $\tau_2$-sentences, and such that for all $\tau$-sentences $\psi_1$ and $\psi_2$, with $\psi_1$ a $\bigdoublewedge$-sentence, if $\sigma_1 \models \psi_1$ and $\sigma_2 \models \psi_2$ then $\psi_1 \wedge \psi_2$ is satisfiable.

In the rest of the proof, we make the convention that if $\sigma \in \Sigma$ and we write $\sigma = \sigma_1 \cup \sigma_2$, then $\sigma_1$ and $\sigma_2$ are the witnesses that $\sigma \in \Sigma$, i.e., $\sigma_1$ consists of $\bigdoublewedge$-$\tau_2$-sentences, $\sigma_2$ consists of $\tau_2$-sentences, and they satisfy the satisfiability condition above.

We claim that $\Sigma$ is a consistency property. The following claim will verify many of the conditions.

\begin{claimstar}
Fix $\sigma \in \Sigma$ and write $\sigma = \sigma_1 \cup \sigma_2$. If $\phi$ is a $\tau_i$-sentence (and a $\bigdoublewedge$-sentence if $i = 1$) with $\sigma_i \models \phi$, then $\sigma \cup \{\phi\} \in \Sigma$.
\end{claimstar}
\begin{proof}
We will show the case $i = 1$. We can write $\sigma \cup \{\phi \} = (\sigma_1 \cup \{\phi\}) \cup \sigma_2$. If $\sigma_1 \cup \{\phi\} \models \psi_1$ and $\sigma_2 \models \psi_2$, with $\psi_1$ a $\bigdoublewedge$-sentence, then since $\sigma_1 \models \phi$, $\sigma_1 \models \psi_1$. Hence $\psi_1 \wedge \psi_2$ is satisfiable.
\end{proof}

\noindent We now check the conditions of a consistency property.
\begin{enumerate}
\item If $\mu \subseteq \sigma$ with $\sigma \in \Sigma$, write $\mu = \mu_1 \cup \mu_2$ and $\sigma = \sigma_1 \cup \sigma_2$ where $\mu_1 \subseteq \sigma_1$ and $\mu_2 \subseteq \sigma_2$. Given $\mu_1 \models \psi_1$ and $\mu_2 \models \psi_2$, we have $\sigma_1 \models \psi_1$ and $\sigma_2 \models \psi_2$; hence $\psi_1 \wedge \psi_2$ is satisfiable. So $\mu \in \Sigma$.
\item If $\phi,\neg \phi \in \sigma = \sigma_1 \cup \sigma_2$, say $\phi,\neg \phi \in \sigma_i$, $\sigma_i \models \phi \wedge \neg \phi$ which is not satisfiable. The other possible case is that $\phi \in \sigma_i$, $\neg \phi \in \sigma_j$, $i \neq j$, in which case $\sigma_i \models \phi$ and $\sigma_j \models \neg \phi$, and $\phi \wedge \neg \phi$ is not satisfiable.
\item This follows from the claim.
\item This follows from the claim.
\item Write $\sigma = \sigma_1 \cup \sigma_2$. We have two cases which are different, depending on whether $\bigdoublevee_{\phi \in X} \phi \in \sigma_1$ or $\bigdoublevee_{\phi \in X} \phi \in \sigma_2$.

First suppose that $\bigdoublevee_{\phi \in X} \phi \in \sigma_2$. Let $\sigma_{2,\phi} = \sigma_2 \cup \{\phi\}$. We claim that for some $\phi \in X$, $\sigma_{2,\phi} \cup \sigma_1 \in \Sigma$. If not, then for each $\phi \in X$ there are $\tau$-sentences $\psi_{2,\phi}$ and $\psi_{1,\phi}$, with $\psi_{1,\phi}$ a $\bigdoublewedge$-sentence, such that $\sigma_{2,\phi} \models \psi_{2,\phi}$ and $\sigma_1 \models \psi_{1,\phi}$, and such that $\psi_{2,\phi} \wedge \psi_{1,\phi}$ is unsatisfiable. So $\psi_{2,\phi} \models \neg \psi_{1,\phi}$. Since
\[ \sigma_2 \models \bigdoublevee_{\phi \in X} \phi \]
we have that
\[ \sigma_2 \models \bigdoublevee_{\phi \in X} \psi_{2,\phi}.\]
On the other hand,
\[ \sigma_1 \models \bigdoublewedge_{\phi \in X} \psi_{1,\phi}. \]
This formula is a $\bigdoublewedge$-sentence as each $\psi_{1,\phi}$ is. Finally,
\[ \bigdoublevee_{\phi \in X} \psi_{2,\phi} \models \neg \bigdoublewedge_{\phi \in X} \psi_{1,\phi} \]
which contradicts that $\sigma \in \Sigma$.

Now suppose that $\bigdoublevee_{\phi \in X} \phi \in \sigma_1$; then $X$ is finite. We begin in a similar way as before. Let $\sigma_{1,\phi} = \sigma_1 \cup \{\phi\}$. We claim that for some $\phi \in X$, $\sigma_{1,\phi} \cup \sigma_2 \in \Sigma$. If not, there for each $\phi \in X$ there are $\tau$-sentences $\psi_{1,\phi}$ and $\psi_{2,\phi}$, with $\psi_{1,\phi}$ a $\bigdoublewedge$-sentence, such that $\sigma_{1,\phi} \models \psi_{1,\phi}$ and $\sigma_2 \models \psi_{2,\phi}$, and such that $\psi_{1,\phi} \wedge \psi_{2,\phi}$ is unsatisfiable. So $\psi_{1,\phi} \models \neg \psi_{2,\phi}$. Since
\[ \sigma_1 \models \bigdoublevee_{\phi \in X} \phi \]
we have that
\[ \sigma_1 \models \bigdoublevee_{\phi \in X} \psi_{1,\phi}.\]
As $X$ is finite, by Lemma \ref{lem:disj} this is equivalent to a $\bigdoublewedge$-sentence. On the other hand,
\[ \sigma_2 \models \bigdoublewedge_{\phi \in X} \psi_{2,\phi} \]
and
\[ \bigdoublevee_{\phi \in X} \psi_{1,\phi} \models \neg \bigdoublewedge_{\phi \in X} \psi_{2,\phi} \]
which contradicts that $\sigma \in \Sigma$.

\item This follows from the claim as $(\forall x) \phi(x) \models \phi(c)$ for all $c \in C$.
\item If $(\exists x) \phi(x) \in \sigma$, then choose $c \in C$ which does not appear in $\sigma$. Suppose that $(\exists x) \phi(x) \in \sigma_1$; the case where $(\exists x)\phi(x) \in \sigma_2$ is similar. We claim that $\sigma \cup \{\phi(c)\} \in \Sigma$. Suppose that $\sigma_1 \cup \{\phi(c)\} \models \psi_1$ and $\sigma_2 \models \psi_2$, where $\psi_1$ is a $\bigdoublewedge$-sentence. Write $\psi_1 = \theta_1(c)$ and $\psi_2 = \theta_2(c)$. We have $\sigma_1 \models \phi(c) \rightarrow \theta_1(c)$, and so since $c$ does not appear in $\sigma_1$, $\sigma_1 \models (\forall x)[\phi(x) \rightarrow \theta_1(x)]$. Similarly, $\sigma_2 \models (\forall x) \theta_2(x)$. Also, $\sigma_1 \models (\exists x) \phi(x)$ and so $\sigma_1 \models (\exists x) \theta_1(x)$. Since $(\exists x) \phi(x) \in \sigma_1$, $\phi(x)$ is a $\bigdoublewedge$-formula. So $(\exists x)\theta_1(x) \wedge (\forall x) \theta_2(x)$ is satisfiable, say in a model $\mc{M}$. Note that the constant $c$ does not appear in the formula $(\exists x)\theta_1(x) \wedge (\forall x) \theta_2(x)$, so we may choose the interpretation of $c$ in $\mc{M}$ such that $\mc{M} \models \theta_1(c)$. Then $\mc{M} \models \theta_1(c) \wedge \theta_2(c)$.
\item let $t$ be a term with no variables and let $c,d \in C$,
\begin{enumerate}
\item This follows from the claim.
\item Suppose $c=t \in \sigma$ and $\phi(t) \in \sigma$. Write $\sigma = \sigma_1 \cup \sigma_2$. Consider $\mu = \sigma \cup \{\phi(c)\} = \sigma_1 \cup \sigma_2 \cup \{ \phi(c) \}$. Suppose $c=t \in \sigma_i$ and $\phi(t) \in \sigma_j$. The case $i = j$ follows from the claim, so we consider the case $i \neq j$. Suppose that $\sigma_i \models \psi_i$ and $\sigma_j \cup \{\phi(c)\} \models \psi_j$. Then $\sigma_i \models c=t \wedge \psi_i$ and $\sigma_j \models c=t \rightarrow \psi_j$, so $c=t \wedge \psi_i \wedge (c=t \rightarrow \psi_i)$ is satisfiable. So $\psi_i \wedge \psi_j$ is satisfiable.
\item Pick $e \in C$ which does not appear in $\sigma = \sigma_1 \cup \sigma_2$. Then if $\sigma_1 \cup \{e = t\} \models \psi_1$ and $\sigma_2 \cup \{e = t\} \models \psi_2$, write $\psi_1 = \theta_1(e)$ and $\psi_2 = \theta_2(e)$. Then since $e$ does not appear in $\sigma_1$ or $\sigma_2$, $\sigma_1 \models \theta_1(t)$ and $\sigma_2 \models \theta_2(t)$. Thus $\theta_1(t) \wedge \theta_2(t)$ is satisfiable. Given a model of $\theta_1(t) \wedge \theta_2(t)$, setting the interpretation of $c$ to $t$, we get a model of $\psi_1 \wedge \psi_2$. So $\psi_1 \wedge \psi_2$ is satisfiable.
\end{enumerate}
\end{enumerate}

Since $\phi_1 \models \phi_2$, $\{\phi_1,\neg \phi_2\} \notin \Sigma$ as otherwise by the Model Existence Theorem there would be a model of $\phi_1 \wedge \neg \phi_2$. By definition of $\Sigma$, there are $\tau$-sentences $\psi_1$ and $\psi_2$, with $\psi_1$ a $\bigdoublewedge$-sentence, such that $\phi_1 \models \psi_1$, $\neg \phi_2 \models \psi_2$, and $\psi_1 \wedge \psi_2$ is not satisfiable. So we have that $\phi_1 \models \psi_1$, $\psi_1 \models \neg \psi_2$, and $\neg \psi_2 \models \phi_2$. Hence $\phi_1 \models \psi_1$ and $\psi_1 \models \phi_2$.

Thus $\psi_1$ is the desired interpolant, except that it may contain constants from $C$. Write $\psi_1 = \theta(\bar{c})$, where $\theta$ is an $\tau$-formula with no constants from $\bar{c}$. Neither $\phi_1$ nor $\phi_2$ contains constants from $C$, and so $\phi_1 \models (\forall \bar{x}) \theta(\bar{x})$ and $(\exists \bar{x}) \theta(\bar{x}) \models \phi_2$. Since $(\forall \bar{x}) \theta(\bar{x}) \models (\exists \bar{x}) \theta(\bar{x})$, we can take $(\forall \bar{x}) \theta(\bar{x})$ as the interpolant.
\end{proof}

We get the following corollary, which is (1) implies (2) of Theorem \ref{thm:main}. Interestingly, when we apply the interpolation theorem in the proof, one of the languages contains the other (i.e., we have $\tau_1 \supseteq \tau_2$ so that $\tau = \tau_1 \cap \tau_2 = \tau_2$). If it were not for our added assumptions on the form of the formulas involved, finding an interpolant would be trivial as we could just take the sentence in the smaller language.

\begin{corollary}\label{cor:both-to-pc}
Let $\mathbb{K}$ be a class of $\tau$-structures closed under isomorphism. If $\mathbb{K}$ is both a $\PC_{\Delta}$-class and $\mc{L}_{\omega_1, \omega}$-elementary, then it is defined by a $\bigdoublewedge$-sentence.
\end{corollary}
\begin{proof}
Let $\tau^* \supseteq \tau$ be an expanded language and let $X$ be a set of first-order sentences such that $\mathbb{K}$ is the class of reducts to $\tau$ of models of $\psi_1 = \bigdoublewedge_{\phi \in X} \phi$. Note that $\psi_1$ is a $\bigdoublewedge$-sentence.

Let $\psi_2$ be an $\mc{L}_{\omega_1, \omega}(\tau)$-sentence defining $\mathbb{K}$. We have that $\psi_1 \models \psi_2$, so by the Interpolation Theorem, there is a $\bigdoublewedge$-$\tau$-sentence $\theta$ such that $\psi_1 \models \theta$ and $\theta \models \psi_2$.

Every $\mc{M} \in \mathbb{K}$ has an expansion which is a model of $\psi_1$ and hence is itself a model of $\theta$; and every model of $\theta$ is a model of $\psi_2$, and hence in the class $\mathbb{K}$. So $\theta$ defines $\mathbb{K}$.
\end{proof}

\section{The Skolem Argument}\label{sec:simple}

For the direction (2)$\Rightarrow$(1) of Theorem \ref{thm:main}, we must prove the following theorem. The proof works for sentences from $\mc{L}_{\kappa,\omega}$ for any $\kappa$.  The essence of this theorem is that a $\bigdoublewedge$-sentence of $\mc{L}_{\kappa, \omega}$ can be Skolemized to a first-order theory.

\begin{theorem}\label{thm:2-1}
	Let $\mathbb{K}$ be a class of structures closed under isomorphism. If $\mathbb{K}$ is defined by a $\bigdoublewedge$-theory of $\mc{L}_{\kappa,\omega}$, then it is a pseudo-elementary (PC$_{\Delta}$) class.
\end{theorem}
\begin{proof}
	We will argue by induction on the complexity of formulas that for every $\bigdoublewedge$-formula $\varphi(\bar{x})$, there is a collection of elementary first-order formulas $\Phi(\varphi)$ in an expanded language $\tau_\varphi$ such that for every $\mc{A} \in \mathbb{K}$ there is an expansion $\mc{A}^+$ of $\mc{A}$ to $\tau_\varphi$ such that, for all $\bar{a} \in A$, we have
	$$\mc{A} \models \varphi(\bar{a}) \iff \forall \varphi' \in \Phi(\varphi), \mc{A}^+ \models \varphi'(\ba)$$
	
	Indeed, suppose that $\varphi(\bar{x})$ is a $\bigdoublewedge$-formula. Then either $\varphi(\bar{x})$ is already an elementary first-order formula (in which case $\Phi(\varphi) = \{\varphi\}$), or it is of one of the following forms:
	\begin{enumerate}
		\item $\varphi \equiv (\exists y) \psi(\bar{x},y)$.
		
		Take $\Phi(\varphi) = \{\theta(\bar{x},f_\theta(\bar{x})) \mid \theta \in \Phi(\psi) \}$ where the $f_\theta$ are new function symbols.
		
		\item $\varphi \equiv (\forall y) \psi(\bar{x},y)$.
		
		Take $\Phi(\varphi) = \{ (\forall y) \theta(\bar{x},y) \mid \theta \in \Phi(\psi) \}$.
		
		\item $\varphi \equiv \bigdoublewedge_{i \in I} \psi_i(\bar{x})$.
		
		Take $\Phi(\varphi) = \bigcup_{i\in I} \Phi(\psi_i)$
	\end{enumerate}
Then showing that $\Phi(\varphi)$ has the desired property is easy by induction. For $\varphi$ in case (2) or (3), the expansion $\mc{A}^+$ of $\mc{A}$ is inherited from the subformulas. For $\varphi$ in case (1), in addition to the inherited expansion, the $f^{\mc{A}^+}_\theta(\bar{a})$ still need to be defined. For $\bar{a}$ such that $\mc{A} \models \varphi(\bar{a})$, choose $b$ such that $\mc{A} \models \psi(\bar{a},b)$, and set  $f^{\mc{A}^+}_\theta(\bar{a})=b$ for all $\theta \in \Phi(\psi)$. For $\bar{a}$ such that $\mc{A} \not \models \varphi(\bar{a})$, choose any $b \in A$ and set  $f^{\mc{A}^+}_\theta(\bar{a})=b$ for all $\theta \in \Phi(\psi)$.

	
	Note that the languages constructed can be assumed to be disjoint for different sentences, so there is no extra difficulty in starting with a theory, rather than a sentence.
\end{proof}

\section{Game Formulas}\label{sec:game}

In this section, we show how the direction (2)$\Rightarrow$(1) of Theorem \ref{thm:main} follows from known results on game formulas.

\begin{definition}
	A closed game formula is an expression of the form
	\[ \forall y_1 \exists z_1 \forall y_2 \exists z_2 \cdots \bigdoublewedge_{n} \varphi_n(\bar{x},y_1,z_1,y_2,z_2,\ldots)\]
	where each $\varphi_n$ is an elementary first-order formula. Such a formula is computable if the sequence $\varphi_n$ is computable.
\end{definition}

\noindent Satisfaction for such formulas is defined by a game played between two players, with player I playing the $\forall$ quantifiers and player II playing the $\exists$ quantifiers; player II wins, and the formulas is satisfied, if he can make $\varphi_n(\bar{x},y_1,z_1,\ldots)$ true for every $n$.  Alternatively, satisfaction can be defined by the existence of Skolem functions (which turn out to be the winning strategies for player II).

Note that each $\varphi_n$ has finitely many free variables.  Also, the `closed' adjective refers to use of conjunctions in the formula.

Every (computable) $\bigdoublewedge$-formula is equivalent to a (computable) closed game formula by moving all of the quantifiers to the front. So we can get we get the direction (2)$\Rightarrow$(1) of Theorem \ref{thm:main} as well as Theorem \ref{thm:comp} as corollaries of the following theorems, respectively.

\begin{theorem}\cite[Theorem 2.1.4]{Kolaitis}\label{thm:Kolaitis}
 Any class of $\tau$-structures defined by a closed game formula is $\PC_\Delta$.
\end{theorem}
\begin{theorem}\cite[Corollary 6.6.7]{Barwise}
		Any class of $\tau$-structures defined by a computable closed game formula is $\PC'$. 	
\end{theorem}

\noindent The proof given in the previous section is, however, much simpler.  Indeed, the proof in Section \ref{sec:simple} gives a proof of Theorem \ref{thm:Kolaitis} because the Skolem functions for closed game formulas are still finitary functions because each stage of the game has only finitely many plays before it (and because each of the formulas $\varphi_n$ has finitely many free variables).  This proof could be further generalized to consider longer games, showing that any class defined by a higher analogue of closed game formulas is $\PC$ in some infinitary logic $\mc{L}_{\kappa, \lambda}$.

\bibliography{References}
\bibliographystyle{alpha}

\end{document}